\documentclass[11pt,a4paper,twoside]{amsart}
\usepackage{url} 
\usepackage[latin1]{inputenc}
\usepackage[T1]{fontenc}
\usepackage[english]{babel}
\usepackage{pifont,fancybox,pstricks,pst-grad,pst-node}
\usepackage{amssymb,amscd,amstext,latexsym,array,verbatim,txfonts,graphicx,array,colortbl}
\usepackage[pdfusetitle, pdfpagelabels, plainpages=false,
  bookmarks, bookmarksnumbered]{hyperref} 
\usepackage{bookmark}
\usepackage{cite}%
\hypersetup{pdfauthor={Mathieu Dutour Sikiri\'c, Herbert Gangl, Paul E. Gunnells, Jonathan Hanke, Achill Sch\"urmann, and Dan Yasaki}}
\newcommand{\tpdf}{\texorpdfstring}

\theoremstyle{plain}
\newtheorem{thm}{\bfseries Theorem}[section]

\newtheorem{prop}[thm]{\bfseries Proposition}
\newtheorem{cor}[thm]{\bfseries Corollary}
\newtheorem{df}[thm]{\bfseries Definition}
\theoremstyle{remark}

\DeclareMathSymbol{\Z}{\mathalpha}{AMSb}{"5A} 
\DeclareMathSymbol{\PP}{\mathalpha}{AMSb}{"50} 
\DeclareMathSymbol{\Q}{\mathalpha}{AMSb}{"51}
\DeclareMathSymbol{\N}{\mathalpha}{AMSb}{"4E}
\DeclareMathSymbol{\R}{\mathalpha}{AMSb}{"52}
\def\C{\mathbb{C}}

\newcommand{\longto}{\longrightarrow}


\definecolor{mongris}{gray}{0.9}

\renewcommand{\bf}{\bfseries}

\DeclareMathOperator{\Vor}{Vor}

\def \St{{S\hskip-1.5pt t}}

\newcommand{\GL}{\mathrm{GL}}

\newcommand{\Unitary}{\mathrm {U}}
\newcommand{\Stab}{\mathrm{Stab}}
\def \tr {\textcolor{black}}
\def \tg {\textcolor{black}}

\newif\ifshowvc
\showvctrue
\showvcfalse  

\ifshowvc
\input{vc}
\fi

\definecolor{mongris}{rgb}{0.9, 0.9, .9}

\usepackage[all]{xy}
\renewcommand{\geq}{\geqslant}
\renewcommand{\leq}{\leqslant}

\newcommand{\cS}{\mathcal{S}}

\begin{document}

\author[M. D. Sikiri\'c]{Mathieu Dutour Sikiri\'c}
\address{Mathieu Dutour Sikiri\'c, Rudjer Boskovi\'c Institute, Bijenicka 54, 10000 Zagreb, Croatia}
\email{mathieu.dutour@gmail.com}

\author[H. Gangl]{Herbert Gangl}
\address{Department of Mathematical Sciences, South Road, Durham DH1 3LE, United Kingdom}
\email{herbert.gangl@durham.ac.uk}

\author[P. E. Gunnells]{Paul E. Gunnells}
\address{P. E. Gunnells, Department of Mathematics and Statistics, LGRT 1115L, University of Massachusetts, Amherst, MA 01003, USA}
\email{gunnells@math.umass.edu}

\author[J. Hanke]{Jonathan Hanke}
\address{J. Hanke, Princeton, NJ 08542, USA}
\email{jonhanke@gmail.com}
\urladdr{\url{http://www.jonhanke.com}}

\author[A. Sch\"urmann]{Achill Sch\"urmann}
\address{A. Sch\"urmann, Universit\"at Rostock, Institute of Mathematics, 18051 Rostock, Germany}
\email{achill.schuermann@uni-rostock.de}

\author[D. Yasaki]{Dan Yasaki}
\address{D. Yasaki, Department of Mathematics and Statistics, University of North Carolina at Greensboro, Greensboro, NC 27412, USA}
\email{d\_yasaki@uncg.edu}

\thanks{MDS was partially supported by the Humboldt Foundation.
PG was partially
supported by the NSF under contract DMS 1101640 and DMS 1501832.
The authors thank the American Institute of Mathematics where this research was initiated.
}



\title{On the topological computation of \tpdf{$K_4$}{K4} of the
Gaussian and Eisenstein integers}

\begin{abstract}
In this paper we use topological tools to investigate the structure of
the algebraic $K$-groups $K_{4} (\Z[i])$ and $K_{4} (\Z[\rho])$, where
$i := \sqrt{-1}$ and $\rho := (1+\sqrt{-3})/2$.  We exploit the close
connection between homology groups of $\GL_n(R)$ for $n\leq 5$ and
those of related classifying spaces, then compute the former using
Voronoi's reduction theory of positive definite quadratic and
Hermitian forms to produce a very large finite cell complex on which
$\GL_n(R)$ acts.  Our main result is that $K_{4} (\Z[i])$ and $K_{4}
(\Z[\rho])$ have no $p$-torsion for $p\geq 5$.  
\end{abstract}

\maketitle

\section{Introduction} \subsection{Statement of results} Let $R$ be
the ring of integers of a number field $F$.  Only very few cases are
known where the algebraic $K$-group $K_4(R)$ has been explicitly
computed, the first such $K_4(\Z)$ having been determined as recently
as 2000 by Rognes \cite{RognesK4}, building on work of Soul\'e
\cite{soule-3torsion}.  The goal of this paper is the explicit
topological computation of the torsion (away from $2$ and $3$) in the
groups $K_{4} (R)$ for $R$ one of two special imaginary quadratic
examples: the \emph{Gaussian integers} $\Z[i]$ and the
\emph{Eisenstein integers} $\Z[\rho]$, where $i := \sqrt{-1}$ and
$\rho := (1+\sqrt{-3})/2$.  Our work is in the spirit of Lee--Szczarba
\cite{LS_Hom_Cohom, LS4, LS3_fortyeight}, Soul\'e \cite{SouleSL3}, and
Elbaz-Vincent--Gangl--Soul{\'e} \cite{EGS_quelques, PerfFormModGrp}
who treated $K_N(\Z)$ for small $N$, and Staffeldt \cite{k3gauss} who
investigated $K_{3}(\Z[i])$.  As in these works, the first step is to
compute the cohomology of $\GL_n( R)$ for $n\leq N+1$; information
from this computation is then assembled into information about the
$K$-groups following the program in \S \ref{ss:quillen}.  Using these
computations we show the following (Theorem \ref{boundedprimes}):

\medskip
\noindent{\bf \tr{Theorem.}} {\em
The orders of the groups \ $K_{4} \big(\Z[i]\big)$ \ 
and \ $K_{4} \big(\Z [\rho]\big)$ 
are not divisible by any primes $p\geq 5$.}

\medskip
We remark that this result is not new; in fact, Kolster's work
\cite{Kolster1} implies the stronger result that $K_{4}
\big(\Z[i]\big)$ and $K_{4} \big(\Z [\rho]\big)$ vanish.  Indeed, if
$R$ is the ring of integers of a $CM$ field, then Kolster proved that,
assuming the Quillen--Lichtenbaum conjecture, the orders of the groups
$K_{4n} (R)$, $n=1,2,3,\dotsc$, can be computed in terms of special
values of certain $L$-functions.  This deep connection between
$K$-groups and special values of $L$-functions is now a theorem,
thanks to the celebrated work by Voevodsky \cite{vvbk} and Rost, as
put into context in \cite{wbook}.

Our work, on the other hand, treats $K_{4} \big(\Z[i]\big)$ and $K_{4}
\big(\Z [\rho]\big)$ by completely different methods.  We only use the
definition of the $K$-groups and explicit results about the cohomology
of the relevant arithmetic groups \cite{AIM1}, together with
Arlettaz's bounds on the kernel of the Hurewicz homomorphism
\cite{ArlettazJPAA71}, to prove Theorem \ref{boundedprimes}.  This
also explains why our calculations do not allow us to say anything for
the primes 2 and 3: both the results of \cite{AIM1} and the
injectivity of the Hurewicz map in our cases only hold away from these
primes.

\subsection{Outline of method}\label{ss:quillen} In the rest of this
introduction we outline the main steps of our argument.  These follow
the classical approach for computing algebraic $K$-groups of number
rings due to Quillen \cite{QuillenFiniteGen}, which shifts the focus
to computing the homology (with nontrivial coefficients) of certain
arithmetic groups.

\begin{enumerate}
\item[(i)] {\bf (Definition)} By definition the algebraic $K$-group
$K_N( R)$ of a ring $R$ is a particular homotopy group of a
topological space associated to $R$: we have $K_N( R)=\pi_{N+1}(BQ(
R))$, where $BQ( R)$ is a certain classifying space attached to the
infinite general linear group $\GL( R)$.  In particular $BQ (R)$ is
the classifying space of the category $Q( R)$ of finitely generated
$R$-modules.  This is known as Quillen's \emph{$Q$-construction} of
algebraic $K$-theory \cite{QuillenHigherK}.
\item[(ii)] {\bf (Homotopy to Homology)} The Hurewicz homomorphism
$\pi_{N+1}(BQ( R))\to H_{N+1}(BQ( R))$ allows one to replace the
homotopy group by a homology group without losing too much
information; more precisely, what may get lost is information about
small torsion primes appearing in its finite kernel.
\item[(iii)] {\bf (Stability)} By a stability result of Quillen
\cite[p.~198]{QuillenFiniteGen} one can pass from $Q(R)$ to the
category $Q_{\tg{M+1}}( R)$ of finitely generated $R$-modules of rank $\leq
\tg{M}+1$ for sufficiently large $\tg{M}$. This amounts to passing from $\GL(
R)$ to the finite-dimensional general linear group $\GL_{\tg{M}+1}( R)$.
\tg{In the cases at hand, a result of Lee and Szczarba allows to reduce
to the case $M=N$.
}
\item[(iv)] {\bf (Sandwiching)} The homology groups to be determined are
then $H_*(BQ_n(R))$ for $n\leq  N+1$.
Rather than \tr{computing} these directly, one uses the fact that they can be
sandwiched between homology groups of
$\GL_n(R)$, where the homology is taken with
(nontrivial) coefficients in the Steinberg module $\St_n$ associated
to $\GL_n(R)$.
\item[(v)] {\bf (\tg{Equivariant} homology)} 
\tg{It has been shown for certain number rings $R$ that the homology groups $H_m(\GL_{n}(R), \St_{n})$ 
are isomorphic to the equivariant $\GL_{n}(R)$-homology
of an associated pair (denoted $(X_n^*,\partial X_n^*)$ in  \S\ref{outline} below). 
The standard method to compute the latter uses {\em Voro\-noi complexes}.  }
These are \tg{relative} chain complexes of
certain explicit polyhedral reduction domains of a space of positive
definite quadratic or Hermitian forms of a given rank, depending
respectively on whether $R = \Z$ or $R$ is imaginary quadratic.  
%
\item[(vi)] {\bf (Vanishing Results)} 
There are various techniques to show vanishing of homology groups.  
As a starting point one has vanishing results 
for $H_n(BQ_1)$ as in Theorem \ref{cornerstone} below, 
and for $H_0(GL_n,\St_n)$ \tr{as in Lee--Szczarba \cite{LS_Hom_Cohom}, Cor.~to
Thm 4.1.} 

\end{enumerate}
For a given $N$, using (ii) and knowing the results of (iv)--(vi) for
all $0 \leq n\leq N+1$ is often enough to give a bound $p \leq B$ on
the primes $p$ dividing the order of the torsion subgroup $K_{N,
\mathrm{tors}}(R)$ of $K_{N}(R)$.  

\subsection{Outline of paper}\label{outline} In this paper the sections work
backwards through the method outlined in \S\ref{ss:quillen} to
determine the structure of $K_4(\Z[i])$ and $K_4(\Z[\rho])$.  In
\S\ref{s:homology}, we describe the computation of the 
\tg{equivariant homology in question and relate it to the Steinberg homology.}
In
\S\ref{s:vanishingandsandwiching} we use the \tg{results on Steinberg} homology and some
vanishing results to determine the groups $H_m(BQ_n({}R))$ (i.e.,~step
(iv) above).  A key role here is played by Quillen's stability result
(iii) for $BQ_n$, \tr{as refined by Lee-Szczarba in \cite{LS_Hom_Cohom}}, which serves as a stopping criterion.  Finally, in
\S \ref{s:hurewicz} we work out the potential primes entering the
kernel of the Hurewicz homomorphism (i.e.,~step (ii) above), which
gives Theorem \tr{\ref{boundedprimes}}. 

\section{Homology of Voronoi complexes}\label{s:homology}

We first collect the results from \cite{AIM1} concerning the Voronoi
complexes attached to $\Gamma=\GL_n(\Z[i])$ or $\Gamma=\GL_n(\Z[\rho
])$; this is the necessary information needed for step (v) from \S
\ref{ss:quillen} above.  More details about these computations,
including background about how the computations are performed, can be
found in \cite{AIM1}.

\medskip

Let $F$ be an imaginary quadratic field with ring of integers $R$, and
let $X_{n} := \GL_{n} (\C) / \Unitary (n)$ be the symmetric space of
$\GL_{n} (F \otimes_{\Q} \R)$.  The space $X_{n}$ can be realized as
the quotient of the cone of rank $n$ positive definite Hermitian
matrices $C_{n}$ modulo homotheties (i.e. non-zero scalar
multiplication), and a partial Satake compactification $X^{*}_{n}$ of
$X_{n}$ is given by adjoining boundary components to $X_{n}$ given by
the cones of positive semi-definite Hermitian forms with an
$F$-rational nullspace (again taken up to homotheties).  We let
$\partial X_n^* := X_{n}^{*}\smallsetminus X_{n}$ denote the
\emph{boundary} of $X_{n}^*$.  Then $\Gamma := \GL_n(R)$ acts by left
multiplication on both $X_n$ and $X_n^*$, and the quotient $\Gamma
\backslash X^{*}_{n}$ is a compact Hausdorff space.

A generalization---due to Ash \cite[Chapter II]{amrt} and Koecher
\cite{koecher}---of the polyhedral reduction theory of Voronoi
\cite{Voronoii} yields a $\Gamma$-equivariant explicit decomposition
of $X_{n}^{*}$ into (Voronoi) cells.  Moreover, there are only
finitely many cells modulo $\Gamma$ \tg{and we have the following result.}

\begin{prop}  [{\cite[Proposition 3.6]{AIM1}}]
\label{thm:voriso}\tr {For $\Gamma\in \{\GL_n(\Z[i]), \GL_n(\Z[\rho])\}$ and $m\in \Z$ 
we have} $\tg{H_{m}^\Gamma (X_n^*,\partial X_n^*,\Z) \simeq H_{m-n+1} (\Gamma, \St_{n})}$.
\end{prop}

Let
$\Sigma_d^*:=\Sigma_d(\Gamma)^*$ be a set of representatives of the
$\Gamma$-inequivalent $d$-dimensional Voronoi cells that meet the
interior $X_{n}$, and let $\Sigma_d:=\Sigma_d(\Gamma)$ be the subset of
representatives of the $\Gamma$-inequivalent {\em orientable} cells in
this dimension; here we call a cell {\em orientable} if all the
elements in its stabilizer group preserve its orientation.
\tr{Note that in our consideration the prime~2 will always be inverted.
This entails that only orientable cells can contribute to the homology.}
One can form a chain complex $\Vor_{*}$, the
\emph{Voronoi complex}, and one can prove that modulo small primes the
homology of this complex is the homology $H_{*} (\Gamma , \St_n)$,
where $\St_n$ is the rank $n$ \emph{Steinberg module}
(cf.~\cite[p.~437]{Borel-Serre}).
To keep track of these small primes explicitly, we make the following definition.

\begin{df}[Serre class of small prime power groups]\label{def:serreclass} Given $k \in \N$,
we let $\cS_{p \leq k}$ denote the \tr{\bf Serre class} of finite abelian
groups $G$ whose cardinality $|G|$ has all of its prime divisors $p$
satisfying $p \leq k$.

For any finitely generated abelian group $G$, there is a unique
maximal subgroup $G_{p \leq k}$ of $G$ in the Serre class $\cS_{p \leq
k}$.  We say that two finitely generated abelian groups $G$ and $G'$
are {\bf equivalent modulo $\cS_{p\leq k}$} and write $G \simeq_{{/p
\leq k}} G'$ if the quotients $G/G_{p \leq k} \cong G'/G'_{p \leq k}$
are isomorphic.
\end{df}
\tr{We call the {\bf torsion primes} of a group $G$ those prime numbers $p$ which divide the order of at least one of the finite subgroups of $G$.}


\subsection{Voronoi data for \tpdf{$R=\Z[i]$}{R = Z[i]}} We now give
results for the Voronoi complexes and \tg{the equivariant homology of the pairs
$(X_n^*,\partial X_n^*)$} in the cases
relevant to our paper \tg{($n=2,3,4$)}.  This subsection treats the Gaussian integers;
in \S \ref{ss:voronoieisenstein} we treat the Eisenstein integers.


\begin{prop}[{\cite{k3gauss}}]
\leavevmode
\begin{enumerate}
\item
There is one $d$-dimensional Voronoi cell for $\GL_2(\Z[i])$ for each $1 \leq d \leq 3$, 
and only the 3-dimensional cell is orientable.

\item
The number of $d$-dimensional Voronoi cells for $\GL_3(\Z[i])$ is given by:
\[
\begin{array}{|c|ccccccc|}
\hline
d & 2& 3& 4&  5&6&7&8 \\
\hline
|\Sigma_d(\GL_3(\Z[i]))^*| &   2& 3& 4& 5& 3& 1& 1 \\
\hline
|\Sigma_d(\GL_3(\Z[i]))| &   0& 0& 1& 4& 3& 0& 1
\\\hline
\end{array}
\]
\end{enumerate}
\end{prop}

\phantom{.}
\medskip
\begin{prop}[{\cite[Table 12]{AIM1}}]
The number of $d$-dimensional Voronoi cells for $\GL_4(\Z[i])$ is given by:
\[
\begin{array}{|c|ccccccccccccc|}
\hline
d &  3& 4&  5&6&7&8 &9 &10&11&12&13&14&15 \\
\hline
|\Sigma_d(\GL_4(\Z[i]))^*| &    4& 10& 33& 98& 258& 501& 704& 628& 369& 130& 31& 7& 2 \\
\hline
|\Sigma_d(\GL_4(\Z[i]))| &  0& 0& 5& 48& 189& 435& 639& 597& 346& 120& 22& 2& 2\\ 
\hline
\end{array}
\]
\end{prop}

We remark that for $\GL_3(\Z[i])$ the Voronoi complexes and their
homology ranks were originally computed by Staffeldt \cite{k3gauss},
who even distilled the 3-part for each homology group.  After
calculating the differentials for this complex one obtains the following homology groups,
in agreement with Staffeldt's results:

\begin{prop}[{\cite[Theorems IV, 1.3 and 1.4, p.785]{k3gauss}}]
\mbox{}
\begin{equation} \label{homGL2i}
H_m(\GL_2(\Z[i]),\St_2) \simeq_{/p\leq 3} \begin{cases}\ \Z &\text{if } m=2,\\ \ 0&\text{otherwise},\end{cases}
\end{equation}
\begin{equation} \label{homGL3i}
H_m(\GL_3(\Z[i]),\St_3) \simeq_{/p\leq 3} \begin{cases}\ \Z &\text{if } m=2,3,6,\\ \ 0&\text{otherwise}.\end{cases}
\end{equation}
\end{prop}
\noindent
In particular, from the above theorem we deduce that the only possible
torsion primes for $\,H_m(\GL_n(\Z[i]),\St_n)\,$ for $n=2,3$ are
the primes 2 and 3.

\medskip 

While the Voronoi homology of $\GL_4(\Z[i])$ has been determined in all degrees in 
{\cite[Theorem 7.2]{AIM1}}, we will only need the following \tg{two} special cases.

\tr{\begin{prop}[{\cite[Theorem 7.2]{AIM1}}] \label{HmGL4} For $m=1,2$ we have
\begin{equation} \label{HmGL4results}
H_m(\GL_4(\Z[i]),\St_4) 
\simeq_{/p\leq 5} \{0\}\,.
\end{equation}
\end{prop}
}

\noindent
The last column of \cite[Table 12]{AIM1} further shows that 
the elementary divisors of all the differentials in
the Voronoi complex for $\GL_4(\Z[i])$ in small degree (in fact for degree $\leq 13$) 
are supported on primes $\leq 3$. 



\smallskip 
\tr{We want to show the stronger result that $H_1(\GL_4(\Z[i]),\St_4)\simeq_{/p\leq 3} \{0\}$, \tg{i.e.~we
want to show that the prime~5 cannot occur.}
For this we will need to use spectral sequences.}
\tr{According}\footnote{\tr{More precisely [5, VII.7] constructs a
spectral sequence converging to the equivariant homology $H^G_*(X, M)$
of a $G$-complex $X$ with coefficients in a $G$-module $M$; the $E^1$
page has a form similar to \eqref{eq:E1}.  One can formulate an analogous
spectral sequence for the equivariant homology of a pair $(X,Y)$ of
$G$-complexes with $E^1$ page \eqref{eq:E1}, cf.~the remarks in [5, VII.7] in
the paragraphs preceding equation (7.1).}} to \cite[VII.7]{Brown},
there is a spectral sequence $E_{d,q}^r$ converging to the equivariant
homology groups $H_{d+q}^{\Gamma} (X_n^* , \partial X_n^* ; {\mathbb
Z})$ of the homology pair $(X_n^* , \partial X_n^*)$, and such that
\begin{equation}\label{eq:E1}
E_{d,q}^1 = \bigoplus_{\sigma \in \Sigma_d^*} H_q (\Gamma_{\sigma} , {\mathbb Z}_{\sigma}) ,
\end{equation}
where ${\mathbb Z}_{\sigma}$ is the orientation module of the cell
$\sigma$ \tg{and $\Gamma_\sigma$ the stabilizer of the cell $\sigma$.
In the remainder of this section we put $n=4$ and consider $(X_4^*,\partial X_4^*)$}.

\begin{prop} \label{H1GL4a} Let $\Gamma=\GL_4(\Z[i])$ and $E_{d,q}^1 $ as above.\\
(i) For each $d=0,\dots,4$ one has 
$ \ E_{d,4-d}^1 
\simeq_{/p\leq 3} \{0\}$. \\ 
(ii) Similarly, for each $d=0,\dots,5$ one has 
$ \ E_{d,5-d}^1 
\simeq_{/p\leq 3} \{0\}$. 
\end{prop}
\noindent{\em Proof.} \tr{We use the data obtained in \cite[Table 12]{AIM1}, available at \cite{dany}.}\\
(i) 1. As there are no cells in $\Sigma_d^*$ for $d\leq 2$, we have
$E_{0,4}^1=E_{1,3}^1=E_{2,2}^1=0$.\\
2. Consider now $d=3$.  \tr{The stabilizer of each of the four cells in $\Sigma_3^*$ lies in $\cS_{p\leq 3}$.}
Thus in particular we have
\[
E_{3,1}^1 = \bigoplus_{\sigma\in \Sigma_3^*} H_1(\Stab_\sigma,\Z_\sigma) \in \cS_{p\leq 3},
\]
where $\cS_{p\leq 3}$ is as in Definition \ref{def:serreclass}.\\
\tr{3. For $d=4$, 
we note that none of the ten cells in $\Sigma_4^*$ has its orientation preserved under the action of
its stabilizer, so $E_{4,0}^1 = 0 \text{ mod }{\cS_2}$.} \\
(ii) 1. As there are no cells in $\Sigma_d^*$ for $d\leq 2$, we have
$E_{0,5}^1=E_{1,4}^1=E_{2,3}^1=0$.\\
2. Consider now $d=3$ and $d=5$.  \tr{The stabilizer of each of the four cells in $\Sigma_3^*$ and each of the 
33 cells in $\Sigma_5^*$  lies in $\cS_{p\leq 3}$.}
Thus in particular we have
\[
E_{3,2}^1  \in \cS_{p\leq 3},\qquad E_{5,0}^1 \in \cS_{p\leq 3}\,.
\]
\tr{3. Finally, for $d=4$, 
 there is only one cell (out of ten) in
$\Sigma_4^*$, denoted by $\sigma_4^1$, that contains a subgroup of
order~5.  We must therefore show that there is no 5-torsion in
$\tr{H_1}(\Stab(\sigma_4^1),\tilde\Z)$  (where $\tilde\Z$ is the orientation module 
$\Z_{\sigma_4^1}$).
 Indeed, the
subgroup $K_1$ of $\Stab(\sigma_4^1)$ preserving the orientation of $\sigma_4^1$ 
is isomorphic to $\Z /4\Z \times A_5$, 
where $A_{5}$ is the alternating group on five
letters, with abelianization
$\tr{H_1}(\Stab(\sigma_4^1),\tilde\Z)\simeq \tr{H_1}(K_1,\Z)$ 
$\simeq  \Z /4\Z $ (for the first equality, which  holds mod $\cS_2$, we make use of Lemmas 8.2 and 8.3 in \cite{PerfFormModGrp})
lies in $\cS_{p\leq 3}$.  
Thus there can be no $5$-torsion from here, which completes the proof. 
\qquad $\Box$}
\begin{cor} For $\Gamma=\GL_4(\Z[i])$ 
one has $\ H_1(\Gamma, \St_{\tr{4}}) \ \tr{\simeq} \ \tg{H_4^\Gamma(X_4^*,\partial X_4^*,\Z)}  \simeq_{/p\leq 3} \{0\}$ and
$\ H_2(\Gamma, \St_{\tr{4}}) \ \tr{\simeq} \ \tg{H_5^\Gamma(X_4^*,\partial X_4^*,\Z)}   \simeq_{/p\leq 3} \{0\}$.
\end{cor}


\subsection{Voronoi homology data for \tpdf{$R=\Z[\rho]$}{R = Z[rho]}}\label{ss:voronoieisenstein}

Now we turn to the Eisenstein case.

\begin{prop}[{\cite[Tables 1 and 11]{AIM1}}]
\mbox{}
\begin{enumerate}
\item
There is one $d$-dimensional Voronoi cell for $\GL_2(\Z[\rho])$ for each $1 \leq d \leq 3$, 
and only the 3-dimensional cell is orientable.

\smallskip
\item
The number of $d$-dimensional Voronoi cells for $\GL_3(\Z[\rho])$ is given by:
\[
\begin{array}{|c|ccccccc|}
\hline
d & 2& 3& 4&  5&6&7&8 \\
\hline
|\Sigma_d(\GL_3(\Z[\rho]))^*| &   1& 2& 3& 4& 3& 2& 2 \\
\hline
|\Sigma_d(\GL_3(\Z[\rho]))| &   0& 0& 1& 2& 1& 1& 2
\\\hline
\end{array}
\]

\item
The number of $d$-dimensional Voronoi cells for $\GL_4(\Z[\rho])$ is given by:
\[
\begin{array}{|c|ccccccccccccc|}
\hline
d &  3& 4&  5&6&7&8 &9 &10&11&12&13&14&15 \\
\hline
|\Sigma_d(\GL_4(\Z[\rho]))^*| &    2& 5& 12& 34& 82& 166& 277& 324& 259& 142& 48& 15& 5 \\
\hline
|\Sigma_d(\GL_4(\Z[\rho]))| &  0& 0& 0& 8& 50& 129& 228& 286& 237& 122& 36& 10& 5\\ 
\hline
\end{array}
\]

\end{enumerate}
\end{prop}

\bigskip
After calculating the differentials we find the same results as for the homology of $\Z[i]$ above:

\begin{prop}[{\cite[Theorems 7.1 and 7.2 with Propositions 3.2 and 3.6]{AIM1}}]
\mbox{}
\begin{equation} \label{homGL2rho}
H_m(\GL_2(\Z[\rho]),\St_2) \simeq_{/p\leq 3} 
\begin{cases}\ \Z &\text{if } m=2,\\ \ 0&\text{otherwise},\end{cases}
\end{equation}
\begin{equation} \label{homGL3rho}
H_m(\GL_3(\Z[\rho]),\St_3) \simeq_{/p \leq 3} 
\begin{cases}\ \Z &\text{if } m=2,3,6,\\ 
\ 0&\text{otherwise},\end{cases}
\end{equation}
\item For $m=1,2$ we have
\tr{\begin{equation} \label{homGL4rho}
H_m(\GL_4(\Z[\rho]),\St_4) \simeq_{/p\leq 5}\{0\}\,.
\end{equation}}
\end{prop}

As with $\Z[i]$, a more refined analysis of the $\Gamma=GL_4(\Z[\rho])$ case
shows that $H_m^\Gamma(X_4^*,\partial X_4^*,\Z)$ contains no $5$-torsion for $m=4,5$:

\begin{prop} \label{H1GL4} Let $\Gamma=\GL_4(\Z[\rho])$ and $E_{d,q}^1 $ as above.\\
(i) For each $d=0,\dots,4$ one has 
$ \ E_{d,4-d}^1 
\simeq_{/p\leq 3} \{0\}$. \\ 
(ii) Similarly, for each $d=0,\dots,5$ one has 
$ \ E_{d,5-d}^1 
\simeq_{/p\leq 3} \{0\}$. 
\end{prop}

\tg{
\noindent{\em Proof.} The argument is very similar to that of the proof of Proposition
\ref{H1GL4a}. 
\tr{We use the data obtained in \cite[Table 11]{AIM1}, available at \cite{dany}.}\\
(i) 1. As there are no cells in $\Sigma_d^*$ for $d\leq 2$, we have
$E_{0,4}^1=E_{1,3}^1=E_{2,2}^1=0$.\\
2. For $d=3$, there are two cells in $\Sigma_3^*$, with stabilizer in $\cS_{p\leq 3}$, and hence
\[
E_{3,1}^1 = \bigoplus_{\sigma\in \Sigma_3^*} H_1(\Stab(\sigma),\Z_\sigma) \in \cS_{p \leq 3}.
\]
{3. For $d=4$, 
we note that none of the five cells in $\Sigma_4^*$ has its orientation preserved under the action of
its stabilizer, so $E_{4,0}^1 = 0 \text{ mod }{\cS_2}$.} \\
(ii) 1. As there are no cells in $\Sigma_d^*$ for $d\leq 2$, we have
$E_{0,5}^1=E_{1,4}^1=E_{2,3}^1=0$.\\
2. Consider now $d=3$ and $d=5$.  \tr{The stabilizer of each of the two cells in $\Sigma_3^*$ and each of the 
12 cells in $\Sigma_5^*$  lies in $\cS_{p\leq 3}$.}
Thus in particular we have
\[
E_{3,2}^1  \in \cS_{p\leq 3},\qquad E_{5,0}^1 \in \cS_{p\leq 3}\,.
\]
\tr{3. Finally, for $d=4$, 
 there is only one cell (out of five) in
$\Sigma_4^*$, denoted by $\sigma_4^1$, that contains a subgroup of
order~5.  We must therefore show that there is no 5-torsion in
$\tr{H_1}(\Stab(\sigma_4^1),\tilde\Z)$  (where $\tilde\Z$ is the orientation module 
$\Z_{\sigma_4^1}$). 
 Indeed, the
subgroup $K_1$ of $\Stab(\sigma_4^1)$ preserving the orientation of $\sigma_4^1$ 
is isomorphic to $\Z /6 \Z \times A_5$, 
where $A_{5}$ is the alternating group on five
letters, with abelianization
$\tr{H_1}(\Stab(\sigma_4^1),\tilde\Z)=\tr{H_1}(K_1,\Z)$ 
$\simeq  \Z /6\Z $, which
lies in $\cS_{p\leq 3}$.  
Thus there can be no $5$-torsion from here, which completes the proof. 
\qquad $\Box$}
\begin{cor} For $\Gamma=\GL_4(\Z[\rho])$ 
one has $H_1(\Gamma, \St_{\tr{4}}) \ \tr{\simeq} \ H_4^\Gamma(X_4^*,\partial X_4^*,\Z)  \simeq_{/p\leq 3} \{0\}$ and
$\ H_2(\Gamma, \St_{\tr{4}}) \ \tr{\simeq} \ H_5^\Gamma(X_4^*,\partial X_4^*,\Z)   \simeq_{/p\leq 3} \{0\}$.
\end{cor}
}

\section{Vanishing and sandwiching}\label{s:vanishingandsandwiching} 

In this section, we carry out the sandwiching argument (step (iv) of
\S \ref{ss:quillen}).  As a first step we
invoke a vanishing result for homology groups for $BQ_1$ due to
Quillen \cite[p.212]{QuillenFiniteGen}.  In our cases this result boils down
to the following statement:

\begin{prop} 
\label{cornerstone} For the rings $R=\Z[i]$ and $\Z[\rho]$, we
have $$H_n\big(BQ_1\tr{(R)}\big)=0\qquad \text{whenever\ } n\geq 3\,.$$
\end{prop}
For $R=\Z[i]$ a slightly stronger result is proved in \cite[Lemma I.1.2]{k3gauss}. 
However, we will not need this stronger result for $\Z [i]$, or
its analogue for $\Z [\rho]$.

Using our homology data from \S \ref{s:homology} and Proposition~\ref{cornerstone}, we can get for both rings $R=\Z[i]$ and
$R=\Z[\rho]$ the following result:

\begin{prop} \label{H5BQ} $H_5\big(BQ\tr{(R)}\big) \simeq_{/p\leq 3} \Z$. 
\end{prop}

\begin{proof}
For brevity we will drop $R$ from the notation, as the argument is the same for both cases.
We will successively determine $H_5(BQ_j)$ for $j=1,\dots,5$ and then
identify the last group via stability with $H_5(BQ)$.  For this, we
will combine results from \S \ref{s:homology} with Quillen's long
exact sequence for different \tr{$j$}, given by
\begin{equation} \label{Quillenexact}
\cdots\longto H_n(BQ_{\tr{j}-1}) \longto H_n(BQ_{\tr{j}}) \longto H_{n-\tr{j}}(\GL_{\tr{j}},\St_{\tr{j}})\longto H_{n-1}(BQ_{\tr{j}-1})
 \longto\cdots  .
\end{equation}

\medskip
{\em The case $j=1$.} By Proposition \ref{cornerstone} we have  $H_n(BQ_1)=0$ for $n\geq 3 $.

\medskip
{\em The case $j=2$.} From the above sequence \eqref{Quillenexact} for $\tr{j}=2$, we get
$$\underbrace{H_5(BQ_{1})}_{=0} \longto H_5(BQ_2) \longto H_{3}(\GL_2,\St_2)\longto \underbrace{H_4(BQ_{1})}_{=0}\,,$$
whence $H_5(BQ_2) =0 \mod \cS_{p \leq 3}$ by \eqref{homGL2i} and \eqref{homGL2rho}.

\medskip
{\em The case $j=3$.} Now we invoke another result of Staffeldt, who showed  (see \cite[proof of Theorem I.1.1]{k3gauss} 
 that 
\begin{equation} \label{H4BQ}
H_4(BQ_2) =H_4(BQ_3) =\Z \mod \cS_{p \leq 3}\,.
\end{equation}

From  \eqref{Quillenexact} for $\tr{j}=3$ we get the exact sequence, working mod $\cS_{p \leq 3}$, 
$$\hskip -10pt H_5(BQ_{2}) \longto H_5(BQ_3) \longto
\underbrace{H_{2}(\GL_3,\St_3)}_{=\Z  \text{ \ (by \eqref{homGL3i}, \eqref{homGL3rho})}}
\longto \underbrace{H_4(BQ_{2})}_{=\Z \text{ \ (by \eqref{H4BQ})}}
\longto \underbrace{H_4(BQ_3)}_{=\Z \text{ \ (by \eqref{H4BQ})}}
\longto \underbrace{H_{1}(\GL_3,\St_3)}_{=0 \text{ \ (by
\eqref{homGL3i}, \eqref{homGL3rho})}}\,.$$

Since the leftmost group $H_5(BQ_{2})$ vanishes modulo $\cS_{p \leq 3}$ by the
case $j=2$, this sequence implies that $H_5(BQ_3)=\Z \mod \cS_{p \leq 3}$.

\medskip {\em The case $j=4$.} Moreover, since
$H_2(\GL_4,\St_4)=H_1(\GL_4,\St_4) =0\mod \cS_{p \leq 3}$ by Proposition
\ref{HmGL4} and Propositions \ref{H1GL4a} and \ref{H1GL4}, the sequence
\eqref{Quillenexact} for $\tr{j}=4$ gives in a similar way that
\begin{equation}\label{H5BQ4}
H_5(BQ_4)=H_5(BQ_3)=\Z \mod \cS_{p \leq 3}\,.
\end{equation}

\medskip
{\em The case $j=5$.}  This is the most complicated of all the cases
to handle.  {Note that $BQ$ is an $H$-space which implies that $H_*(BQ)\otimes \Q$ is the enveloping algebra of $\pi_*(BQ)\otimes \Q$.
It is well-known that $K_0(\Z[i])=\Z$, $K_1(\Z[i])=\Z/2$ and $K_2(\Z[i])=0$ \cite[Appendix]{bass-tate}  
as well as $K_3(\Z[i])=\Z\oplus\Z/{24}$ (given by Merkurjev--Suslin, cf.~e.g.~Weibel \cite{WeibelHandbook}, Theorem 73 in combination with Example 28), so modulo $\cS_{p \leq 3}$ we have 
$$\pi_1 (BQ )\otimes \Q = K_0(\Z[i]) \otimes \Q= \Q\,,$$
as well as  $\,\pi_2(BQ)\otimes \Q =\pi_3(BQ)\otimes \Q =0$, and 
$$\pi_4(BQ) \otimes \Q = K_3(\Z[i])\otimes \Q  = \Q\,.$$
A very similar argument works for $\Z[\rho]$.\\
Hence $H_5(BQ)\otimes \Q $ contains the product  of $\pi_1(BQ)\otimes \Q $ by $\pi_4(BQ)\otimes \Q $ and so
its dimension is at least 1. 

The stability result foreshadowed in step (iii) of \S \ref{ss:quillen}
\tr{(resulting for a Euclidean domain $\Lambda$ from
$H_0(\GL_n(\Lambda),\St_n)=0$ for $n\geq 3$, \cite[Corollary to
Theorem 4.1]{LS_Hom_Cohom})}, now implies that one has $\
H_5(BQ)=H_5(BQ_5)\,$.  By the above we get that the rank of
$H_5(BQ_5)=H_5(BQ)$ is at least 1.  }

Therefore, invoking yet again Quillen's exact sequence \eqref{Quillenexact}, this time for $j=5$, and using the 
above result that $H_5(BQ_4)$ is equal to $\Z$ modulo $\cS_{p \leq 3}$, we deduce from
$$\underbrace{H_5(BQ_{4})}_{=\Z\ \text{by }\eqref{H5BQ4}} \longto
H_5(BQ_5) \longto \underbrace{H_{0}(\GL_5,\St_5)}_{=0}$$ that
$H_5(BQ)=H_5(BQ_5)$ must be equal to $\Z$ modulo $\cS_{p \leq 3}$ as well.
Thus $H_{5} (BQ)$ cannot contain any $p$-torsion with $p>3$.
\end{proof}

\section{Relating \tpdf{$K_4(\tr{R})$}{K4(O)} and \tpdf{$H_5(BQ(\tr{R}))$}{H5(BQ(O))} via the
Hurewicz homomorphism}\label{s:hurewicz} 

It is well known that for a number ring $\,R\,$ the space $BQ( R)$ is
an infinite loop space.  Hence a theorem due to Arlettaz \cite[Theorem
1.5]{ArlettazJPAA71} shows that the kernel of the corresponding
Hure\-wicz homomorphism $K_4( R) =\pi_5(BQ)\to H_5(BQ)$ is certainly
annihilated by $144$ (cf.~Definition~1.3 in loc.cit., where this
number is denoted $R_{5}$).  Thus that kernel lies in $\cS_{p \leq 3}$
(Definition \ref{def:serreclass}).

Therefore this Hurewicz homomorphism is injective modulo $\cS_{p \leq 3}$. For
$\,R=\Z[i]$ or $\Z[\rho]$, Proposition \ref{H5BQ} implies that
$H_5(BQ)$ contains no $p$-torsion for $p>3$.  After invoking Quillen's
result that $K_{2n}({}R)$ is finitely generated and Borel's result
that the rank of $K_{2n}({}R)$ is zero for any number ring $R$ and
$n>0$, we obtain the following theorem:

\begin{thm} \label{boundedprimes}
The groups $K_{4} (\Z [i])$ and $K_{4} (\Z [\rho])$ lie in $\cS_{p
\leq 3}$.
\end{thm}

\medskip
{\bf Acknowledgments.} 
 We thank Ph.~Elbaz-Vincent for very helpful discussions.  We also thank an 
anonymous referee for suggesting numerous improvements and corrections 
to our paper.  This research was conducted
as part of a ``SQuaRE'' (Structured Quartet Research Ensemble) at the
American Institute of Mathematics in Palo Alto, California in September
2013.  It is a pleasure to thank AIM and its staff for their support,
without which our collaboration would not have been possible.

\bibliographystyle{amsplain_initials_eprint}
\bibliography{k_theory}

\end{document}